\documentclass{article}
\usepackage{amsfonts,amsmath,amsthm,amscd,amssymb,latexsym}
\usepackage{graphicx}
\DeclareGraphicsExtensions{.eps}

\begin{document}
\title{Non-paradoxical action of automata groups on infinite words}
\author{Victoriia Korchemna}
\thanks {The author expresses her thanks to A. S. Oliynyk for introducing her to the topic of automata transformations}
\date{28.10.2018} 
\theoremstyle{plain}
\newtheorem{theorem}{Theorem}
\newtheorem{lemma}{Lemma}
\newtheorem{proposition}{Proposition}
\newtheorem{definition}{Definition}
\newtheorem{remark}{Remark}
\maketitle 
		
\begin{abstract}
	 We show, that groups, defined by wide class of automata,  including all polynomial ones, act on the set of infinite words not paradoxical.\\
\end{abstract}

\section{Introduction and definitions} 
\subsection{Rooted tree of words and it's isomorphisms}

 Let $X$ be a finite set, which will be called alphabet with elements called letters. We always suppose $|X| > 1$ (here $|X|$ denotes the cardinality of $X$). Let $X^*$ be the free monoid generated by $X$. The elements of this monoid are finite words $x_1x_2 ... x_n$, $x_i \in X$, including the empty word $\emptyset$. Denote by $X^w$ the set of all infinite words $x_1x_2 ... x_n ...$, $x_i \in X$. For each $l \in \mathbb{N}$ and $w=w_1...w_l... \in X^w$ set $w[:l]:=w_1...w_l$. In other words, $w[:l]$ is a finite word, formed by first $l$ letters of $w$.
 
 The set $X^*$ is naturally a vertex set of a rooted tree, in which two words are connected by an edge if and only if they are of the form $v$ and $vx$, where $v\in X^* $, $x \in X.$ The empty word $\emptyset$ is the root of the tree $X^*$.
 
 A map $f:X\to X$ is an endomorphism of the tree $X$, if for any two adjacent vertices $v$, $vx\in X^*$ the vertices $f(v)$ and $f(vx)$ are also adjacent, so that there exist $u \in X^*$ and $y \in X$ such that $f(v) = u$ and $f(vx) = uy$. An automorphism is a bijective endomorphism.
 
 \subsection{Automata and automorphisms of rooted trees}
 An automaton A is a quadruple $(X,Q,\pi,\lambda)$, where:
 \begin{itemize}
 	\item $X$ is an alphabet;
    \item $Q$ is a set of states of the automaton;
    \item $\pi : Q \times X \to X$ is a map, called the transition function of the automaton;
    \item $\lambda : Q \times X \to X$ is a map, called the output function of the automaton.
 \end{itemize}   
 
 An automaton is finite if it has a finite number of states. The maps $\pi,\lambda$ can be extended on $Q \times X^*$ by the following recurrent formulas:
 $$\pi(q,\emptyset) = q,\;\;\pi(q,xw) = \pi(\pi(q,x),w)$$
 $$\lambda(q,\emptyset) = \emptyset,\;\;\lambda(q,xw) = \lambda(q,x)\lambda(\pi(q,x),w),$$
 where $x \in X$, $q \in Q$, and $w \in X^*$ are arbitrary elements. Similarly, the maps $\pi,\lambda$ are
 extended on $Q \times X^w$.
 
 An automaton $A$ with a fixed state $q$ is called initial and is denoted by $A_q$.
 Every initial automaton defines the automorphism $\lambda(q,\cdot)$ of the rooted tree $X^*$, which we also denote by $A_q(\cdot) = \lambda(q,\cdot)$ (or q($\cdot$) if it is clear, which automaton it belongs to). We denote by $e$ a trivial state of automaton, i.e., such a state that defines a trivial automorphism of $X^*$. The action of
 an initial automaton $A_q$ can be interpret as the work of a machine, which being in the state $q$ and reading on the input tape a letter $x$, goes to the state $\pi(q,x)$, types on the output tape the letter $\lambda(q,x)$, then moves both tapes to the next position and proceeds further.
 
 \subsection{Moore diagrams}
 An automaton $A$ can be represented (and defined) by a labelled directed graph,
 called the Moore diagram, in which the vertices are the states of the automaton and
 for every pair $(q,x) \in Q \times X $ there is an edge from $q$ to $\pi(q,x)$ labelled by $x|\lambda(q,x)$.
 \begin{center} 	
 	\includegraphics[width=8cm, trim=0 450 0 300,clip]{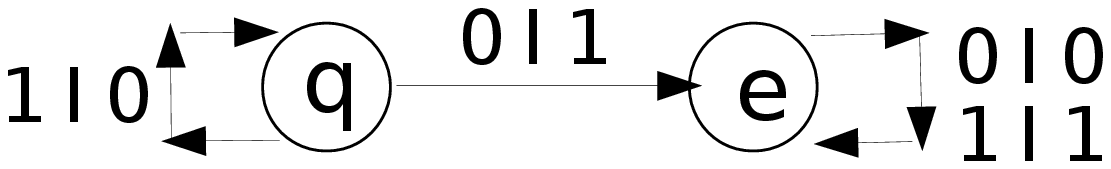} 
 \end{center}
 Here is the Moore diagram of automaton, called the adding machine. 
 Consider a word of length $l \in \mathbb N$ as a binary number, with lower digits on the left side. If the automaton gets the word in state q, it adds 1 modulo $2^l$ to it.
 
 \subsection{Inverse automaton. Composition of automata}
 An automaton is called invertible, if for each $q \in Q$ the mapping $\lambda(q,\cdot)$ is a bijection.
 Suppose that we have the Moore diagram of invertible automaton $A$. Let us swap all the left labels with right ones that correspond to them. After renaming the states $q \to q^{-1}$ we get a Moore digram of some automaton $A^{-1}$, which is called the inverse automaton of $A$. Notice that for each $q \in Q$ the state $q^{-1}$ of $A^{-1}$  defines the automorphism of a rooted tree which is inverse to $q$.
 
 Further in the article we omit the word "invertible" and consider only invertible automata.
 
 By giving the output of $A = (X, Q, \pi, \lambda)$ to the input of another automaton $B = (X, S, \alpha, \beta)$,
 we get an application which corresponds to the automaton called the composition of $A$ and $B$ and is denoted by $A*B$. This automaton is formally described as the automaton with the set of the states $Q \times S$ and the transition and output functions $\phi, \psi$ defined by
 $$\phi((q, s), x) = (\pi(q, x), \alpha(s, \lambda(q, x)))$$
 $$\psi((q, s), x) = \beta(s, \lambda(q, x))$$
 Notice that 
 a state $(q,s)$ of $Q\times S$  defines the automorphism of a rooted tree which is a superposition of ones defined by $q$ and $s$.
 
\subsection{Paradoxical actions of groups}
Let $G$ be a group acting on a set $X$ and suppose $E \subseteq X$. $G$ acts on $E$ paradoxically ($E$ is $G$-paradoxical) if for some positive integers $m,n$ there are pairwise disjoint subsets $A_1,...,A_n$, $B_1,...,B_n$ of $E$ and $g_1,...,g_n,\;h_1,...,h_n \in G$ such that $E=\bigcup g_i(A_i)=\bigcup h_i(B_i).$
Loosely speaking, the set $E$ has two disjoint subsets, each of which can be taken apart and rearranged via $G$ to cover all of $E$. Well known example is Banach-Tarski paradox, where subgroup of isometries of $R^3$ act's on a sphere.

\newpage
\section{Automata, that "almost always" move to the trivial state}

In the article we fix some alphabet $X$ and consider words over $X$ only. For arbitrary automata transformation $g$ denote by $NS(g,l)$ the number of words of length $l$ which move $g$ to the non-trivial state. 

\begin{proposition}
$NS(\bullet, \bullet)$ has the following properties:\\	
(1):  $NS(g,l)=NS(g^{-1},l)$\\
(2):  $NS(gh,l)\le NS(g,l) + NS(h,l)$,\\
where $g,h$ are arbitrary automata transformations, $l \in \mathbb{N}$
\end{proposition}

\begin{proof}
$(1)$ Let $w$ be a word with a length $l$ that doesn't move $g$ to the trivial state. Then $g(w)$ is a word with a length $l$ that doesn't move $g^{-1}$ to the trivial state. Since $g$ is a bijective mapping on $X^l$, $g(w)$ are different for different $w$. Therefore, $NS(g^{-1},l) \le NS(g,l).$. After replacing $g \to g^{-1}$, we get the opposite inequality.

$(2)$ For arbitrary automata transformation $q$ denote by $\textbf {S}(q,l)$ the set of words of length $l$ that move $q$ to the trivial state; $\textbf {NS}(q,l):=X^l \setminus \textbf {S}(q,l)$. Then $|\textbf {NS}(q,l)|= NS(q,l)$, $|\textbf {S}(q,l)|= |X|^l-NS(q,l)$.

As $g$ is a bijection on $X^l$, we have $|g(\textbf {NS}(g,l))|=|\textbf {NS}(g,l)|= NS(g,l)$. If a word $w\in X^l$ moves $g$ to the trivial state and so does $g(w)$ with $h$, then $w$ moves $gh$ to the trivial state. Therefore: 
$$\textbf {S}(gh,l) \supseteq g(\textbf {S}(g,l)) \cap \textbf{S}(h,l)=X^l \setminus (g(\textbf {NS}(g,l)) \cup  \textbf {NS}(h,l))$$
$$|\textbf {S}(gh,l)| \ge |X|^l-(NS(g,l)+NS(h,l))$$
$$NS(g,l)= |\textbf {NS}(gh,l)| \le NS(g,l)+NS(h,l)$$
\end{proof}
 
Denote by $G_0$ the set of all automata transformations $g$ for which the following holds: $$NS(g,l)=o(|X|^l), l \to \infty$$ To put it simply, "almost" all the words move the transformations to the trivial state.
 
$(1)$ and $(2)$ imply that $G_0$ is a group. Notice that $G_0$ includes all the polynomial automata. What do we know about it's action on $X^w$?
          
\begin{theorem}
	$X^w$ is not $G_0$-paradoxical.
\end{theorem}

\begin{proof}
Assume that $X^w$ is $G_0$-paradoxical. Then there are the elements  $h_1, ..., h_d$ of $G_0$ and the partition $X^w=A_1\sqcup...\sqcup A_d$ with such property:
\textit{
Let initially every word from $X^w$ have 1 coin. Then $h_i$ moves a coin from every $a_i \in A_i$ to $h_i(a_i)$, $1 \le i \le d$. After that every word from $X^w$ has at least 2 coins.}

For arbitrary $l,s \in \mathbb{N}$, $s \ge 8$ consider the set of $s\cdot|X|^l$ consecutive words\\ $F:=\{w+1,..., w+s|X|^l\}\subset X^w$ and its superset of $(s+2)|X|^l$ consecutive words\\ $F':=\{w+1-|X|^l,..., w+(s+1)|X|^l\}$, put $\neg F':=X^w\setminus F'$. Split $F$ into subsets of $|X|^l$ consecutive words $F=F_1 \sqcup...\sqcup F_s$. Notice that in every such subset the beginnings (first $l$ letters) of words form the set $X^l$.
Fix an arbitrary $i\in\{1...d\}$. Denote $g_i=h_i^{-1}$ The number of words, which are moved by  $g_i$ from $F$ to $\neg F'$ is not greater then $s \cdot NS(h_i,l)$. Actually, consider arbitrary $F_j \subset F$, $1\le j \le s$. If for some $v\in F_j$, $g_i(v)$ is in $\neg F'$, then $g_i$ changes a letter with position number greater then $l$ in word $v$. It means that if $g_i$ receives $v[:l]$ it doesn't move to trivial state after reading it. As $\{v[:l]|v \in F_j\}=X^l$, number of such words in $F_j$ is not greater then $NS(g_i,l)=NS(h_i,l)$. So in $F$ this number is not greater then $s \cdot NS(h_i,l)$. 

The result can be reworded: the number of words in $\neg F'$, from which $h_i$ brings coins to $F$, is not greater then $s \cdot NS(h_i,l)$. So, for each $i \in \{1,...,d\}$, $h_i$ brings not more then $s \cdot NS(h_i,l)$ coins to $F$ from $\neg F'$. Then all $h_i$ bring at most $s \cdot \sum_{i=1}^{d} NS(h_i,l)$ coins.  
As $NS(h_i,l)=o(|X|^l), l \to \infty$, we can get $$s \cdot \sum_{i=1}^{d} NS(h_i,l) \le \frac{1}{4}s|X|^l$$ by choosing large $l$. Therefore, if $|F|$ is quiet large, then the number of coins, which are moved from $\neg F'$ to $F$ is not greater then $\frac{1}{4}|F|$. Since $s \ge 8$, $2 \le \frac{1}{4}s$, then at most $(s+2)|X|^l \le \frac{5}{4}s|X|^l=\frac{5}{4}|F|$ coins are moved to $F$ from $F'$ (in particular from $F$). Together, the number of coins at $F$ becomes at most $\frac{5}{4}|F|+\frac{1}{4}|F|=\frac{3}{2}|F|<2|F|$. It means then some words of $F$ get less then 2 coins. Contradiction. 
\end{proof}

\begin{remark} 
 There are transformations, defined by infinite and not polynomial automata, which satisfy conditions of Theorem 1. For example, all non-trivial states of the following automaton define such transformations:

\begin{center}
\includegraphics[width=6cm, angle=270, viewport=100 50 550 1200, clip]{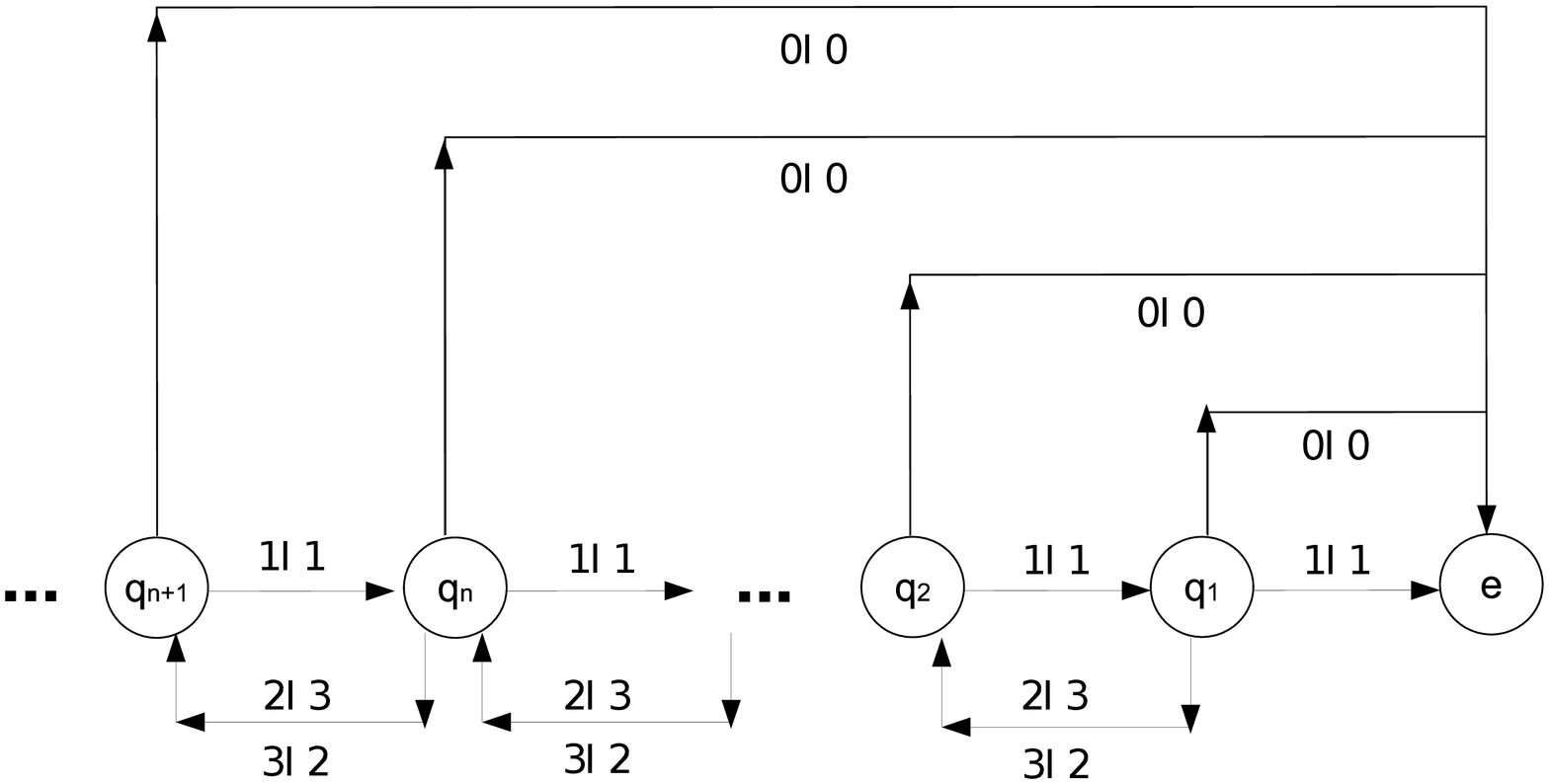}
\end{center}

Each finite word without $0$ and $1$ moves $q_i, i \ge 1$ to non-trivial state. On the other hand, every word that contains $0$ moves them to the trivial state. Therefore, \\ $2^l \le NS(q_i,l) \le 3^l$, where $2^l$ is a number of words without $1$ and $0$, $3^l$ - without $0$ only. Since $\frac{3^l}{4^l}\to 0$, $l\to \infty$, all the states $q_i, i \ge 1$ are acceptable.\\
\end{remark} 

\newpage
\section{Automata, that "almost always" move to cycles}

Theorem 1 doesn't cover some simple automata transformations, which act on $X^w$ non-paradoxically, such as pictured below:
\begin{center}
	\includegraphics[width=2cm, angle=90, trim=100 100 100 100, clip]{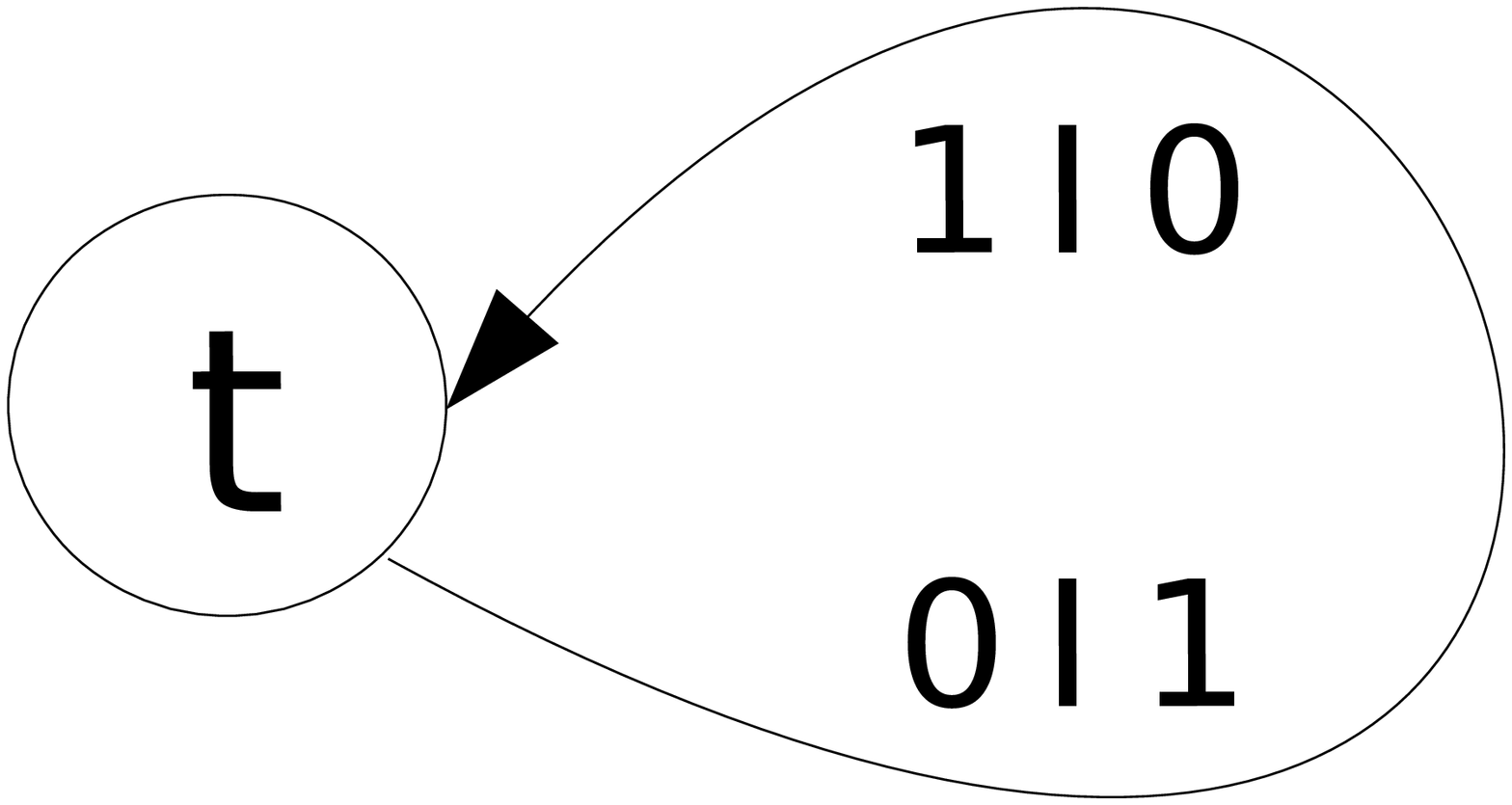} 
\end{center}
In theorem 1, "most" of words must move the automaton to the trivial state. In fact, we can generalise the condition. 

\begin{definition}
	the states $g_1,...,g_n$ of the same automaton form an unconditional cycle (UC), if the transition functions $\pi(g_i,\cdot), 1\le i \le n$ don't depend of input data and have the form $g_1 \to ... \to g_n \to g_1$.
\end{definition}

Apparently, a trivial state is an instance of UC. We will say that a word $w\in X^w \cup X^*$ moves the automata transformation $g$ of the automaton $A$ to a UC in $s$ steps, if $A$, receiving $w$ in $g$, gets into one of the states of the UC after processing first $s$ letters. Notice the following:

\begin{proposition} 
	Assume that $w\in X^l$ moves $g$ to an UC of length $n$ and  $g(w)$ moves $h$ to an UC of length $m$. Then $w$ moves $gh$ to some UC of length $LCM(n,m)$.
\end{proposition} 

\begin{proof} 
We consider $gh$ as a state $(g,h)$ of $A*B$. Let $w$ move $g$ to the UC $g_1,...,g_n$ in $s$ steps and $g(w)$ move $h$ to the UC $h_1,...,h_m$ in $t$ steps. Then after processing first $max\{s,t\}$ letters of $w$ the state $(g,h)$ moves to the UC $(g_1,h_1),...,(g_n,h_m)$ of length $LCM(n,m)$.
\end{proof}  
  
For arbitrary automata transformation $g$ denote by 
$NC_g(l)$ a number of words with length $l$ which don't move $g$ to any UC, and by $\mathbb{NC}(g,l)$ - a set of infinite words that start from them. In other words, $\mathbb{NC}(g,l)\subset X^w$ consists of words that don't move $g$ to any UC while processing the first $l$ letters. 

\begin{proposition} 
$NC(g,l)$ has the following properties:\\
$(1'):\;\;NC(g,l)=NC(g^{-1},l)$\\
$(2'):\;\;NC(gh,l)\le NC(g,l) + NC(h,l)$,\\
where $g,h$ are arbitrary automata transformations, $l \in \mathbb{N}$
\end{proposition} 

\begin{proof} 
 $(1')$ We are going to show that if $w \in X^l$ moves $g$ to the UC $g_1,...,g_n$ then $g(w) \in X^l$ moves $g^{-1}$ to the UC $g_1^{-1},...,g_n^{-1}$. Consider the Moore diagram of some automaton $A$, that contains $g$. Starting from $g$ and moving along left labels that form the word $w$, we achieve the UC $g_1,...,g_n$ in $k \le l$ steps. Let us swap all the left labels with right ones that correspond to them. After renaming the states $h \to h^{-1}$ we get a Moore digram of the automaton $A^{-1}$. This transformation keeps UC, moreover, starting from $g^{-1}$ and moving along left labels that form $g(w)$, we achieve the UC $g_1^{-1},...,g_n^{-1}$ in $k$ steps. So if $w \in X^l$ moves $g$ to the UC $g_1,...,g_n$ then $g(w) \in X^l$ moves $g^{-1}$ to the UC $g_1^{-1},...,g_n^{-1}$. As $g$ is a bijection, we have $NC(g^{-1},l) \le NC(g,l)$. Similarly an opposite inequality can be gotten. 
  
  $(2')$
  Assume that $w\in X^l$ moves $g$ to some UC and $g(w)$ moves $h$ to some (maybe another) UC. Then according to proposition 1 $w$ moves $gh$ to an $UC$. The rest of the proof is similar to one of $(2)$. We just have to replace $NS \to NC$, $\mathbb{NS} \to \mathbb{NC}$, $\mathbb{S} \to \mathbb{C}$.
\end{proof}  

Now we can see that automata transformations $g$ for which $NC(g,l)=o(|X|^l), l \to \infty\\$ form a group. Denote it by $G_1$. Generalising theorem 1, we are going to prove the next statement:

\begin{theorem} 
 $X^w$ is not $G_1$-paradoxical.
\end{theorem} 

We need two auxiliary lemmas. Let us say that a word $w \in X^w$ is $l-$almost periodic if it has the form $w=uv$ where $u \in X^l$ is arbitrary and $v$ is periodic. Mentioning the period, we mean the shortest repeating word from $X^*$, that starts from $(l+1)th$ position of $w$. For example, word $123010101$ as a $3-$almost periodic word has period $01$. But as a $4-$almost periodic word it has period $10$.

\begin{lemma}  
	Let $w$ be $l-$almost periodic word with a period of length $t$. Assume that $w$ moves an automata transformation $g$ to UC of length $c$ in at most $l$ steps. Then $g(w)$ is $l-$almost periodic and the length of it's period divides $LCM(t,c)$.
\end{lemma}

\begin{proof}
When the automaton has already processed the first $l$ letters of $w$, the following holds:
 \begin{itemize}
	\item Every $c$ steps the automaton moves to the same state.
    \item Every $t$ steps the automaton receives the same letter.
 \end{itemize}   
Therefore, every $LCM(t,c)$ steps it receives the same letter in the same state, so gives the same data to the output. 
\end{proof}

Denote by $P_n^l$ the set of all $l-$almost periodic infinite words with the length of period dividing $n!$.\\

\begin{lemma}
	Let $g$ be an automata transformation. Denote by $n(l)$ the maximal length of UC $g$ can move to after processing a word of length $l$. For arbitrary $c\ge n(l)$ there holds:
 $$g(P_c^l \setminus \mathbb{NC}(g,l)) \subseteq P_c^l$$
\end{lemma}

\begin{proof} 
Consider an arbitrary word $w \in P_c^l \setminus \mathbb{NC}(g,l)$. Let the automaton $A$, containing $g$, receive $w$ in the state $g$. After processing the first $l$ symbols by $A_g$, the periodic part of $w$ has already started. Besides that, $A_g$ has already moved to the UC (by definition of $\mathbb{NC}(g,l)$). Since $w \in P_c^l$, then the length of it's period divides $c!$. A length of the UC is not greater then $n(l) \le c$, so it also divides $c!$. According to lemma 1, $g(w)$ is $l-$almost periodic and length of it's period divides $LCM(c!,c!)=c!$. Therefore, $g(w) \in P_c^l$.
\end{proof}

Now we are ready to start the proof of theorem 2.
\begin{proof} 
Assume that $X^w$ is $G_1$-paradoxical. Then there are the elements  $h_1, ..., h_d$ of $G_1$ and the partition $X^w=A_1\sqcup...\sqcup A_d$ with such property:
\textit{
	Let initially every word from $X^w$ has 1 coin. Then $h_i$ moves a coin from every $a_i \in A_i$ to $h_i(a_i)$, $1 \le i \le d$. After that every word from $X^w$ has at least 2 coins.}    

Consider the set $P_N^l$, where $l$ will be defined later and $N=N(l)$ is the maximal length of UC that $h_1,...h_d$ can move to while processing words of length $l$. Fix an arbitrary $i \in \{ 1,...,d \}$. Denote $g_i=h_i^{-1}$. According to lemma 2, $g_i(P_{N}^l \setminus \mathbb{NC}_{g_i(l)}) \subseteq P_{N}^l$. So $g_i$ can move at most $|P_N^l \cap  \mathbb{NC}_{g_i(l)}|$ words from $P_{N}^l$ outside it. Estimate the number. Let $\mathbb{T}_N$ be the set of all periods $T \in X^*$ with length dividing $N!$. For each $T \in \mathbb{T}_N$ denote by $P_{(T)}^l$ the set of $l-$almost periodic infinite words with period $T$. Then we have:
$$P_N^l=\bigsqcup_{T \in \mathbb{T}_N} P_{(T)}^l$$
$$P_N^l \cap  \mathbb{NC}(g_i,l) = \bigsqcup_{T \in \mathbb{T}_N} (P_{(T)}^l \cap  \mathbb{NC}(g_i,l))$$
Notice that for each $v \in X^l$ there is exactly 1 word starting from $v$ in every $P_{(T)}^l ,T \in \mathbb{T}_N$. Therefore $|P_{(T)}^l|=|X|^l$, in particular $$|P_{(T)}^l \cap  \mathbb{NC}(g_i,l)|=NC(g_i,l)$$
$$|P_N^l \cap  \mathbb{NC}(g_i,l)|= \sum_{T \in \mathbb{T}_N} |P_{(T)}^l \cap  \mathbb{NC}(g_i,l)| = |\mathbb{T}_N| \cdot NC(g_i,l)$$
So, $g_i$ moves $|\mathbb{T}_N| \cdot NC(g_i,l)$ words from  $P_{N}^l$ outside it. It means that $h_i$ brings $|\mathbb{T}_N| \cdot NC(g_i,l)$ coins from $X^w \setminus P_N^l$  to  $P_N^l$. Then all $h_1,...,h_d$ bring at most $|\mathbb{T}_N|\sum_{i=1}^d NC(g_i,l)$ coins from $X^w \setminus P_N^l$  to  $P_N^l$. 
As there are $|\mathbb{T}_N| \cdot |X|^l$ words in $P_N^l$ and  $NC(g_i,l)=o(|X|^l), l \to \infty$, similarly to theorem 1 we have that $h_1,...h_d$ don't bring enough coins to $P_N^l$ for large $l$. 
\end{proof}


\newpage
\begin{thebibliography}{5}
 	\bibitem{} Stan Wagon. The Banach-Tarski Paradox. Cambridge University Press, 1985.	
 
 	\bibitem{} R. I. Grigorchuk, V. V. Nekrashevich, V. I. Sushchanskii. Automata, dynamical systems, and groups. Grigorchuk R. I. (ed.), Dynamical systems, automata, and infinite groups. Proc. Steklov Inst. Math. 231 (2000), 128-203.
	
 	\bibitem{} V. Nekrashevych. Self-similar groups. Math. Surveys and Monographs 117. Amer. Math. Soc., Providence, RI, 2005.

 	\bibitem{} I. Bondarenko. Groups Generated by Bounded Automata and Their Schreier Graphs. PhD Dissertation (Texas $A\&M$ Univ., College Station, TX, 2007).
 
 	\bibitem{} Andrzej Zuk. Automata groups.  Paris 7
 
\end{thebibliography}
\end{document}